\documentclass{article}
\usepackage{mathrsfs, amsmath, amsthm, amssymb, bm, mathtools, thm-restate}
\usepackage{graphicx, xcolor, tikz}
\usetikzlibrary{arrows.meta, intersections, calc, shapes}
\usepackage{caption, subcaption}
\usepackage{array}
\usepackage{cases}
\makeatletter
\def\th@plain{%
  \upshape 
}
\makeatother

\makeatletter
\renewenvironment{proof}[1][\proofname]{\par
  \pushQED{\qed}%
  \normalfont \topsep6\p@\@plus6\p@\relax
  \trivlist
  \item[\hskip\labelsep
        \bfseries
    #1\@addpunct{.}]\ignorespaces
}{%
  \popQED\endtrivlist\@endpefalse
}
\makeatother

\newtheorem{theorem}{Theorem}[section]
\newtheorem{lemma}{Lemma}[section]

\newtheorem{conjecture}{Conjecture}
\newtheorem*{conjecture*}{Conjecture}
\newtheorem{case}{Case}

\theoremstyle{definition}

\usepackage[left=25mm, top=25mm, bottom=25mm, right=25mm]{geometry}
\usepackage[T1]{fontenc} 
\usepackage{paralist, tablists}
\usepackage[inline]{enumitem}
\usepackage[square, numbers, sort&compress]{natbib}

\usepackage[pdftex,%
  bookmarks=true,%
  bookmarksnumbered=true, 
  bookmarksopen=true, 
  plainpages=false,%
  pdfpagelabels,%
  colorlinks=true, 
  linkcolor=blue, 
  citecolor=blue,%
  anchorcolor=green,
  urlcolor=blue,
  breaklinks=true,
  hyperindex=true]{hyperref}

\newcounter{Hcase}
\newcounter{Hclaim}

\newcommand{\resetcounter}{\stepcounter{Hcase}\setcounter{case}{0}\stepcounter{Hclaim}\setcounter{claim}{0}}

\newcommand{\etal}{et~al.\ }

\def\int(#1){\mathrm{int}(#1)}
\def\ext(#1){\mathrm{ext}(#1)}
\def\Int(#1){\mathrm{Int}(#1)}
\def\Ext(#1){\mathrm{Ext}(#1)}
\def\ad(#1){\mathrm{ad}(#1)}
\def\mad(#1){\mathrm{mad}(#1)}
\def\la(#1){\mathrm{la}(#1)}

\begin{document}
\title{Star edge-coloring of some special graphs\footnote{All the results are included in the first author's Master Thesis in spring of 2019}}
\author{Xuling Hou { }\quad Lingxi Li {}\quad Tao Wang\footnote{{\tt Corresponding
author: wangtao@henu.edu.cn}}\\
{\small Institute of Applied Mathematics}\\
{\small Henan University, Kaifeng, 475004, P. R. China}}
\date{}
\maketitle
\begin{abstract}

The star chromatic index of a multigraph $G$,  denoted by $\chi_{\mathrm{star}}'(G)$, is the minimum number of colors needed to properly color the edges of $G$ such that no path or cycle of length $4$ is bicolored.

In this paper, we study the star edge-coloring of Halin graphs, $k$-power graphs and the generalized Petersen graphs $P(3n, n)$.
\end{abstract}

\section{Introduction}
All graphs considered in this paper are simple, finite and undirected. We use $V(G)$, $E(G)$, $\Delta(G)$ and $\delta(G)$ to denote the vertex set, edge set, maximum degree and minimum degree of $G$, respectively. A {\bf Halin graph} is a plane graph that consists of a plane embedding of a tree $T$ and a cycle $C$, in which the cycle $C$ connects the leaves of the tree such that it is the boundary of the exterior face and the degree of each interior vertex of $T$ is at least three. The tree $T$ is called the {\bf characteristic tree} and the cycle $C$ is called the {\bf adjoint cycle} of $G$. We write $G = T \cup C$ for a Halin graph with characteristic tree $T$ and adjoint cycle $C$. A {\bf complete Halin graph} is a graph that all leaves of the  characteristic tree are at the same distance from the root vertex.

A {\bf caterpillar} is a tree whose removal of leaves results in a path called the {\bf spine} of the caterpillar. Suppose $G$ is a Halin graph of order $2n+2$ and has a caterpillar $T$ as its characteristic tree where $n \geq 1$. We name the vertices along the spine $P_{n}$ by $1, 2, \dots, n$. The vertices adjacent to $1$ are named by $0$ and $1'$. The vertices adjacent to $h$ are named by $h+1$ and $h'$. Other leaves adjacent to $i$ are named by $i'$, $2\leq i \leq h-1$. Note that $0, 1', \dots, h', h+1$ are vertices lying on the adjoint cycle $C_{h+2}$. If $\{0, 1'\}, \{1', 2'\}, \dots, \{(h-1)', h'\}, \{h', h+1\}$ and $\{h+1, 0\}$ are edges of the adjoint cycle of $G$ (i.e., vertices $0, 1', \dots, h', h+1$ in $C_{h+2}$ are in order), then we can call such a graph $G$ a \textbf{necklace}, which is denoted by $\mathcal{N}_{h}$.

The $k$-power of a graph $G$ is the graph $G^{k}$ whose vertex set is $V(G)$, two distinct vertices being adjacent in $G^{k}$ if and only if their distance in $G$ is at most $k$. Let $P_n=v_1v_2\dots v_n$ be a path of order $n$, we add an edge when the distance between two vertices $v_i, v_j\ (1 \leq i < j \leq n)$ is two, then we can call this graph $P_n^2$. Similarly, $C_n^2$ (i.e. cyclic square graphs) can be defined.

The {\bf generalized Petersen graph} $P(m, n)$ is defined to be the graph of order $2m$ whose vertex set is $\{u_1, \dots , u_m, v_1, \dots , v_m\}$ and edge set is $\{u_iu_{i+1},v_iv_{i+n},u_iv_i\colon i=1,\dots ,m\}$, where the addition is taken modulo $m$.

A proper $k$-edge-coloring of $G$ is a mapping $\phi$: $E(G)\rightarrow\{1, 2, \cdots, k\}$ such that $\phi(e)\neq \phi(e')$ for any two adjacent edges $e$ and $e'$. The chromatic index $\chi'(G)$ of $G$ is the smallest integer $k$ such that $G$ has a proper $k$-edge-coloring.

A proper $k$-edge-coloring $\phi$ of $G$ is called a strong $k$-edge-coloring if any two edges of distance at most two get distinct colors. That is, each color class is an induced matching in the graph $G$. The strong chromatic index, denoted by $\chi_s'(G)$.

A proper $k$-edge-coloring $\phi$ of $G$ is called a star $k$-edge-coloring if there do not exist bichromatic paths or cycles of length four. That is, at least three colors are needed for a path or a cycle of length four. The star chromatic index, denoted by $\chi_{\mathrm{star}}'(G)$.

Liu and Deng \cite{MR2416274} showed that $\chi_{\mathrm{star}}'(G) \leq \lceil 16(\Delta-1)^{\frac{3}{2}}\rceil$ when $\Delta\geq7$. Dvo\v{r}\'{a}k, Mohar, and \v{S}\'{a}mal \cite{MR3019390} presented a near-linear upper bound for  $\chi_{\mathrm{star}}'(G) \leq \Delta \cdot 2^{O(1)\sqrt{\log\Delta}}$. Wang \cite{MR3796353} showed that if a graph $G$ can be edge-partitioned into two graphs $F$ and $H$, then $\chi_{\mathrm{star}}'(G)\leq \chi_{\mathrm{star}}'(F)+\chi_s'(H|_G)$, where $\chi_s'(H|_G)$ denotes the strong chromatic index of $H$ restricted on $G$. In this paper, we use this result to prove the star edge-coloring of $G^k$.

Bezegov\'{a} \cite{MR3431294} \etal proved the following bound for trees.

\begin{theorem}[\cite{MR3431294}] \label{tree}
Let $T$ be a tree with maximum degree $\Delta$. Then $\chi_{\mathrm{star}}'(T) \leq \lfloor 1.5\Delta \rfloor$, and the bound is tight. 
\end{theorem}

Han \etal \cite{MR3924408} showed that the list star chromatic index of tree is also at most $\lfloor 1.5\Delta \rfloor$. 
\begin{theorem}
Let $T$ be a tree with maximum degree $\Delta$. Then $\mathrm{ch}_{\mathrm{star}}'(T) \leq \lfloor 1.5\Delta \rfloor$, and the bound is tight. 
\end{theorem}

Dvo\v{r}\'{a}k, Mohar, and \v{S}\'{a}mal \cite{MR3019390} also studied star edge-coloring of subcubic graphs.

\begin{theorem}[\cite{MR3019390}]
If $G$ is a subcubic graph, then $\chi_{\mathrm{star}}'(G) \leq 7$. 
\end{theorem}

They made the following conjecture.

\begin{conjecture}[\cite{MR3019390}]
If $G$ is a subcubic graph, then $\chi_{\mathrm{star}}'(G) \leq 6$. 
\end{conjecture}

This conjecture is still open. Lu\v{z}ar \cite{MR3904837} provided the current best upper bound for list star chromatic index of subcubic graphs. 

\begin{theorem}
If $G$ is a subcubic graph, then $\mathrm{ch}_{\mathrm{star}}'(G) \leq 7$. 
\end{theorem}

With some restrications on maximum average degree, Lei \etal \cite{MR3818598} gave the following result. 

\begin{theorem}[\cite{MR3818598}]
If $G$ is a subcubic multigraph with $\mad(G) < \frac{24}{11}$, then $\chi_{\mathrm{star}}'(G) \leq 5$. 
\end{theorem}

The upper bound on the maximum average degree was further improved. 
\begin{theorem}[\cite{MR3764342}]
If $G$ is a subcubic multigraph with $\mad(G) < \frac{12}{5}$, then $\chi_{\mathrm{star}}'(G) \leq 5$. 
\end{theorem}

Wang \etal \cite{MR3796353} investigated star edge-coloring of graphs by edge-partition. 
\begin{theorem}[\cite{MR3796353}]\label{partition}
If a graph $G$ can be edge-partitioned into two graphs $F$ and $H$, then $$\chi_{\mathrm{star}}'(G)\leq \chi_{\mathrm{star}}'(F)+\chi_s'(H|_G),$$ where $\chi_s'(H|_G)$ denotes the strong chromatic index of $H$ restricted on $G$.
\end{theorem}
With the aid of \autoref{partition}, they proved the following result on outerplanar graphs.

\begin{theorem}[\cite{MR3796353}]
If $G$ is a outerplanar graph, then $\chi_{\mathrm{star}}'(G) \leq \lfloor 1.5\Delta \rfloor + 5$. 
\end{theorem}

The conjectured best upper bound for the star chromatic index is $\lfloor 1.5\Delta \rfloor + 1$. 

\begin{conjecture}[\cite{MR3431294}]
If $G$ is a outerplanar graph, then $\chi_{\mathrm{star}}'(G) \leq \lfloor 1.5\Delta \rfloor + 1$. 
\end{conjecture}

\begin{theorem}[\cite{Omoomi2018}]\label{PXP}
For two paths $P_{m}$ and $P_{n}$ with  $2 \leq m \leq n$, we have
\[   
\chi'_{\mathrm{star}}(P_m\square P_n) = 
     \begin{cases}
       3, &\quad\text{if $m = n = 2$},\\
       4, &\quad\text{if $m = 2$ and $n \geq 3$}, \\
       5, &\quad\text{if $m = 3$ and $n \in \{3, 4\}$},\\
       6, &\quad\text{otherwise.} \ 
     \end{cases}
\]
\end{theorem}

\section{Star chromatic index of Halin graphs}
In this section, we consider some special Halin graphs. 
\subsection{Star chromatic index of cubic Halin graphs}
\begin{theorem}\label{cubic}
If $G = T \cup C$ is a $3$-regular Halin graph, then $4 \leq \chi_{\mathrm{star}}'(G) \leq 6$. Furthermore, the upper bound is tight.
\end{theorem}
\begin{proof}
Note that every $3$-regular Halin graph contains a triangle. When $|C| = 3$, $G$ is the complete graph $K_{4}$ and $\chi_{\mathrm{star}}'(K_{4}) = 5$. When $|C| \geq 4$, $G$ contains a subgraph isomorphic to the net graph (see \autoref{tuH}). It is easy to see that $\chi_{\mathrm{star}}'(\mathrm{Net}) = 4$. Therefore, $\chi_{\mathrm{star}}'(G)$ has a lower bound four. On the other hand, Dvo{\v{r}}{\'a}k \etal proved that $\chi_{\mathrm{star}}'(H) \geq 4$ for any simple cubic graph $H$ in \cite[Theorem 5.1(b)]{MR3019390}. 

\begin{figure}
\centering
\begin{tikzpicture}
\coordinate (A) at (-2, 0);
\coordinate (B) at (-1, 0);
\coordinate (C) at (1, 0);
\coordinate (D) at (2, 0);
\coordinate (E) at (0, 1.4);
\coordinate (F) at (0, 2.4);
\draw (A)--(D);
\draw (E)--(B);
\draw (E)--(C);
\draw (E)--(F);
\fill (A) circle (2pt);
\fill (B) circle (2pt);
\fill (C) circle (2pt);
\fill (D) circle (2pt);
\fill (E) circle (2pt);
\fill (F) circle (2pt);
\end{tikzpicture}
\caption{Net graph}
\label{tuH}
\end{figure}
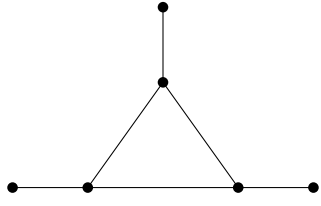

We prove the upper bound by induction on the length $n$ of the adjoint cycle $C$. For $n \in \{3, 4, 5\}$, it is easy to check that $\chi_{\mathrm{star}}'(\mathcal{N}_{1}) = 5$, $\chi_{\mathrm{star}}'(\mathcal{N}_{2}) = 6$, $\chi_{\mathrm{star}}'(\mathcal{N}_{3}) = 5$ (see \autoref{ne123}). Thus, the theorem is true for $n \leq 5$. Next, assume that $n \geq 6$ and the theorem is true for $|C| \leq n - 1$.
\begin{figure}
\centering
\subcaptionbox{Necklace $\mathcal{N}_{1}$\label{N1}}[0.3\linewidth]
{
\begin{tikzpicture}
\foreach \ang in {0, 1, 2}
{
\def\pointname{A\ang}
\coordinate (\pointname) at ($(\ang*360/3+90:1.3)$);
\node[circle, inner sep =1, fill, draw] () at (A\ang) {};}
\draw (A0)--node {\footnotesize 1} (A1)--node {\footnotesize 2}(A2)--node {\footnotesize 3}cycle;
\draw (0, 0)--node {\footnotesize 4}(A0);
\draw (0, 0)--node {\footnotesize 5}(A1);
\draw (0, 0)--node {\footnotesize 1}(A2);
\node[circle, inner sep =1, fill, draw] () at (0, 0) {};
\end{tikzpicture}
}
\subcaptionbox{Necklace $\mathcal{N}_{2}$\label{N2}}[0.3\linewidth]
{
\begin{tikzpicture}
\coordinate (A1) at (-2, 1);
\coordinate (B1) at (2, 1);
\coordinate (A2) at (-2, -1);
\coordinate (B2) at (2, -1);
\coordinate (A3) at (-1, 0);
\coordinate (B3) at (1, 0);
\draw (A1)--node {\footnotesize 5}(B1)--node {\footnotesize 4}(B2)--node {\footnotesize 6}(A2)--node {\footnotesize 4}cycle;
\draw (A1)--node {\footnotesize 1}(A3)--node {\footnotesize 3}(B3)--node {\footnotesize 1}(B1);
\draw (A2)--node {\footnotesize 2}(A3);
\draw (B2)--node {\footnotesize 2}(B3);
\foreach \x in {1, 2, 3}
{\node[circle, inner sep =1, fill, draw] () at (A\x) {};
\node[circle, inner sep =1, fill, draw] () at (B\x) {};}
\end{tikzpicture}
}
\subcaptionbox{Necklace $\mathcal{N}_{3}$\label{N3}}[0.3\linewidth]
{
\begin{tikzpicture}
\coordinate (A1) at (-2, -0.5);
\coordinate (A2) at (-1, 0);
\coordinate (A3) at (0, 0);
\coordinate (A4) at (1, 0);
\coordinate (A5) at (2, -0.5);
\coordinate (B1) at (-1, -1);
\coordinate (B2) at (0, -1);
\coordinate (B3) at (1, -1);
\draw (A1)--node {\footnotesize 5}(A2)--node {\footnotesize 2}(A3)--node {\footnotesize 3}(A4)--node {\footnotesize 4}(A5)--node {\footnotesize 2}(B3)--node {\footnotesize 5}(B2)--node {\footnotesize 4}(B1)--node {\footnotesize 3}cycle;
\draw (A2)--node {\footnotesize 1}(B1);
\draw (A3)--node {\footnotesize 1}(B2);
\draw (A4)--node {\footnotesize 1}(B3);
\draw (A1).. controls+(up:2cm) and + (up:2cm).. node {\footnotesize 1}(A5);
\foreach \x in {1, 2, 3}
\node[circle, inner sep =1, fill, draw] () at (B\x) {};
\foreach \x in {1,...,5}
\node[circle, inner sep =1, fill, draw] () at (A\x) {};
\end{tikzpicture}
}
\caption{}
\label{ne123}
\end{figure}

In the later inductive steps, we use two basic operations to reduce a $3$-regular Halin graph $G$ to another $3$-regular Halin graph $G'$ such that the length of the adjoint cycle of $G'$ is shorter than that of $G$. By the induction hypothesis, $G'$ has a star $6$-edge-coloring. This star $6$-edge-coloring naturally corresponds to a star $6$-edge-coloring of a subgraph $G^{*}$ of $G$. Then, we extend the star edge-coloring of $G^{*}$ to the whole graph $G$, which completes the proof. 

Let $P \coloneqq y_{0}y_{1}y_{2}\dots y_s$ be a longest path in $T$. Since $P$ is a diametral path of $T$, $s \geq 4$ and every neighbor of $y_1$ other than $y_2$ is a leaf. We rename the vertices so that $x = y_{3}$, $w = y_{2}$ and $u = y_{1}$, and let $u_{1}$ and $u_{2}$ be the neighbors of $u$ on $C$ (see \autoref{jubu}).

\begin{figure}
\centering
\begin{tikzpicture}
\coordinate (A0) at (0, 0);
\coordinate (A1) at (90:0.8);
\coordinate (A2) at (230:1.4);
\coordinate (A3) at (-50:1.4);
\coordinate (A4) at ($(A1) + (0, 0.8)$);
\coordinate (A5) at ($(A2) + (-1, 0)$);
\coordinate (A6) at ($(A3) + (1, 0)$);
\coordinate (A7) at ($(A5) + (-1, 0)$);
\coordinate (A8) at ($(A5) + (0, 1)$);
\coordinate (A9) at ($(A6) + (0, 1)$);
\coordinate (A10) at ($(A6) + (1, 0)$);
\coordinate (A11) at ($(A1) + (1, 0)$);

\draw (A0)node[right]{$u$}--(A1)node[left]{$w$}--(A4)node[right]{$x$};
\draw (A0)--(A2)node[below]{$u_{1}$};
\draw (A0)--(A3)node[below]{$u_{2}$};
\draw (A7)--(A5)node[below]{$z_{1}$}--(A6)node[below]{$v_{1}$}--(A10)node[below]{$v_{2}$};
\draw (A5)--(A8);
\draw (A6)--(A9)node[right]{$v$};
\draw (A1)--(A11);
\foreach \x in {0,...,11}
\node[circle, inner sep =1, fill, draw] () at (A\x) {};
\end{tikzpicture}
\caption{Local structure}
\label{jubu}
\end{figure}
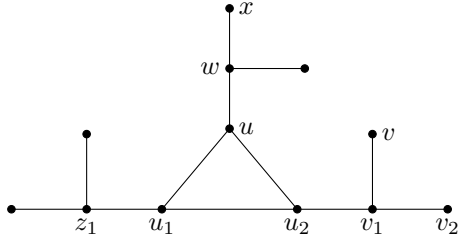

\begin{figure}
\centering
\begin{tikzpicture}
\coordinate (A0) at (0, 0);
\coordinate (A1) at (90:0.8);
\coordinate (A2) at (230:2.5);
\coordinate (A3) at (-50:2.5);
\coordinate (A4) at ($(A2) + (-1, 0)$);
\coordinate (A5) at ($(A3) + (1, 0)$);
\coordinate (A6) at ($(A4) + (-1, 0)$);
\coordinate (A7) at ($(A4) + (0, 1)$);
\coordinate (A8) at ($(A5) + (0, 1)$);
\coordinate (A9) at ($(A5) + (1, 0)$);

\coordinate (B1) at (230:1.6);
\coordinate (B2) at (-50:1.6);
\coordinate (B3) at ($(B1)+(-50:0.9)$);
\coordinate (B4) at ($(B2)+(230:0.9)$);
\draw (A0)node[right]{$w$}--(A1)node[right]{$x$};
\draw (A0)--(B1)node[right]{$u$}--(A2)node[below]{$u_1$};
\draw (A0)--(B2)node[left]{$v$}--(A3)node[below]{$v_{2}$};
\draw (A6)--(A4)node[below]{$z_{1}$}--(B3)node[below]{$u_{2}$}--(B4)node[below]{$v_{1}$}--(A5)node[below]{$z_{2}$}--(A9);
\draw (A4)--(A7);
\draw (A5)--(A8);
\draw (B1)--(B3);
\draw (B2)--(B4);

\foreach \x in {0,...,9}
\node[circle, inner sep =1, fill, draw] () at (A\x) {};
\foreach \x in {1,...,4}
\node[circle, inner sep =1, fill, draw] () at (B\x) {};
\end{tikzpicture}
\caption{$wv\in E(T)$}
\label{wm3}
\end{figure}
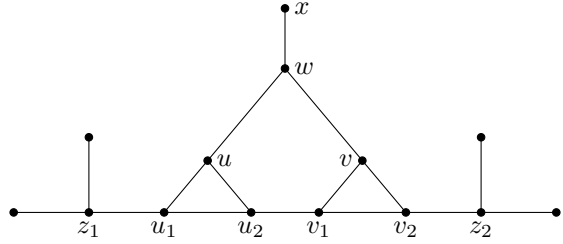

\begin{figure}
\centering
\begin{tikzpicture}
\coordinate (A0) at (0, 0);
\coordinate (A1) at (90:0.8);
\coordinate (A2) at (230:2.5);
\coordinate (A3) at (-50:2.5);
\coordinate (A4) at ($(A2) + (-1, 0)$);
\coordinate (A5) at ($(A3) + (1, 0)$);
\coordinate (A6) at ($(A4) + (-1, 0)$);
\coordinate (A7) at ($(A4) + (0, 1)$);
\coordinate (A8) at ($(A5) + (0, 1)$);
\coordinate (A9) at ($(A5) + (1, 0)$);

\coordinate (B1) at (230:1.6);
\coordinate (B3) at (-50:1.6);
\coordinate (B2) at ($(B1)+(-50:0.9)$);
\coordinate (B4) at ($(B2)+(230:0.9)$);

\draw (A0)node[right]{$w$}--(A1)node[right]{$x$};
\draw (A0)--(B1)node[right]{$u$}--(A2)node[below]{$u_{1}$};
\draw (A0)--(A3)node[below]{$v_{1}$};
\draw (A6)--(A4)node[below]{$z_{1}$}--(B2)node[below]{$u_{2}$}--(A5)node[below]{$v_{2}$}--(A9);
\draw (A4)--(A7);
\draw (A5)--(A8);
\draw (B1)--(B2);

\foreach \x in {0,...,9}
\node[circle, inner sep =1, fill, draw] () at (A\x) {};
\foreach \x in {1,2}
\node[circle, inner sep =1, fill, draw] () at (B\x) {};
\end{tikzpicture}
\caption{$w=v$}
\label{a3}
\end{figure}

Since $w$ has degree three, $w$ has the third neighbor $w_{1}$ other than $x$ and $u$. Without loss of generality, we may assume that $Q$ is the path from $w$ to $v_{1}$ in $T$. Since $P$ is a longest path in $T$, $Q$ has length at most two. It follows that $wv \in E(T)$ or $w = v$. The former implies that $vv_{2} \in E(T)$, and the latter means $wv_{1} \in E(T)$.

\begin{case}
$wv \in E(T)$.
\end{case}
In \autoref{wm3}, let $G'$ be the graph obtained from $G$ by deleting $u, u_{1}, u_{2}, v, v_{1}, v_{2}$, and adding two new edges $wz_{1}$ and $wz_{2}$. By the induction hypothesis, there exists a star $6$-edge-coloring $\varphi$ of $G'$ using colors from the set $\{1, 2, \dots, 6\}$. Without loss of generality, we assume that $\varphi(xw) = 1$, $\varphi(wz_{1}) = 2$, $\varphi(wz_{2}) = 3$. Except the edge $wz_{1}$, let the other two edges incident with $z_{1}$ be colored with $\alpha_{1}$ and $\alpha_{2}$. Except the edge $wz_{2}$, let the other two edges incident with $z_{2}$ be colored with $\beta_{1}$ and $\beta_{2}$.

For each edge $e$ in $E(G') \setminus \{wz_{1}, wz_{2}\}$, let $\psi(e) = \varphi(e)$. Additionly, let $\psi(u_{2}v_{1}) = \varphi(xw) = 1$, $\psi(wv) = \psi(v_{1}v_{2}) = \psi(z_{1}u_{1}) = \varphi(wz_{1}) = 2$, $\psi(wu) = \psi(u_{1}u_{2}) = \psi(v_{2}z_{2}) = \varphi(wz_{2}) = 3$. Until now, there exists no bicolored $4$-path.

In the end, we color these edges $vv_{2}, uu_{1}, uu_{2}, vv_{1}$ in order. For the edge $vv_{2}$, there are at most five forbidden colors $1$, $2$, $3$, $\beta_1$ and $\beta_2$, thus $vv_{2}$ can be colored with a color from $\{1, 2, 3, 4, 5, 6\} \setminus\{1, 2, 3, \beta_1, \beta_2\}$. For the edge $uu_{1}$, there are at most five forbidden colors $1, 2, 3, \alpha_{1}$ and $\alpha_{2}$, thus $uu_{1}$ can be colored with a color from $\{1, 2, 3, 4, 5, 6\} \setminus\{1, 2, 3, \alpha_1, \alpha_2\}$. For the edge $uu_2$, there are at most four forbidden colors $1, 2, 3$ and $\psi(uu_{1})$, thus $uu_{2}$ can be colored with a color from $\{1, 2, 3, 4, 5, 6\} \setminus\{1, 2, 3, \psi(uu_1)\}$. For the edge $vv_{1}$, there are at most five forbidden colors $1, 2, 3, \psi(vv_{2})$ and $\psi(uu_{2})$, thus $vv_{1}$ can be colored with a color from $\{1, 2, 3, 4, 5, 6\}\setminus\{1, 2, 3, \psi (vv_{2}), \psi(uu_2)\}$. It is easy to check that there exists no bicolored $4$-paths.
\begin{case}
$w = v$.
\end{case}
In \autoref{a3}, let $G'$ be the graph obtained from $G$ by deleting $u, u_1, u_2, v_{1}$, and adding two new edges $wz_{1}$ and $wv_{2}$. By the induction hypothesis, there exists a star $6$-edge-coloring $\varphi$ of $G'$ using colors from the set $\{1, 2, 3, 4, 5, 6\}$. Without loss of generality, we may assume that $\varphi(xw)=1$, $\varphi(wz_{1}) = 2$, $\varphi(wv_{2}) = 3$. Except the edge $wz_{1}$, let the other two edges incident with $z_{1}$ be colored with $\alpha_{1}$ and $\alpha_{2}$. Except the edge $wv_{2}$, let the other two edges incident with $v_{2}$ be colored with $\beta_{1}$ and $\beta_{2}$. Except the edge $wx$, let the other two edges incident with $x$ be colored with $\delta_{1}$ and $\delta_{2}$. 

For each edge $e$ in $E(G') \setminus \{wz_{1}, wv_{2}\}$, let $\psi(e) = \varphi(e)$. Additionally, let $\psi(u_{1}u_{2}) = \varphi(xw) = 1$, $\psi(wv_{1}) = \psi(z_{1}u_{1}) = \varphi(wz_{1}) = 2$, $\psi(uu_{1}) = \psi(v_{1}v_{2}) = \varphi(wv_{2}) = 3$. Until now, there is no bicolored $4$-path.

Finally, we color these edges $uw$, $u_{2}v_{1}$ and $uu_{2}$ in order. For the edge $uw$, there are at most five forbidden colors $1$, $2$, $3$, $\delta_{1}$ and $\delta_{2}$, thus $uw$ can be colored with a color from $\{1, 2, 3, 4, 5, 6\}\setminus\{1, 2, 3, \delta_{1}, \delta_{2}\}$. For the edge $u_{2}v_{1}$, we forbid five colors $1, 2, 3, \beta_{1}$ and $\beta_{2}$, thus $u_{2}v_{1}$ can be colored with a color from $\{1, 2, 3, 4, 5, 6\}\setminus\{1, 2, 3, \beta_{1}, \beta_{2}\}$. For the edge $uu_{2}$, there are at most five forbidden colors $1, 2, 3, \psi(u_{2}v_{1})$ and $\psi(uw)$, thus $uu_{2}$ can be colored with a color from $\{1, 2, 3, 4, 5, 6\}\setminus\{1, 2, 3, \psi(u_{2}v_{1}), \psi(uw)\}$. Obviously, there exists no bicolored $4$-paths. 
\resetcounter
\end{proof}

\subsection{Star chromatic index of necklace}
\begin{theorem}
Let $\mathcal{N}_{h}$ be a necklace with $h \geq 1$. If $h$ is odd, then $4 \leq \chi_{\mathrm{star}}'(\mathcal{N}_{h})\leq 5$. 
\end{theorem}
\begin{proof}
It is easy to see that there are at least two triangles in $\mathcal{N}_{h}$ when $h > 1$. So $\mathcal{N}_{h}$ contains a subgraph isomorphic to the Net graph. It follows that $\chi_{\mathrm{star}}'(\mathcal{N}_{h}) \geq \chi_{\mathrm{star}}'(\mathrm{Net}) \geq 4$. 

To prove that $\chi_{\mathrm{star}}'(\mathcal{N}_{h})\leq 5$ if $h$ is odd, it suffices to give a star $5$-edge-coloring for $\mathcal{N}_{h}$ with the coloring set $\{a, b, c, d, f\}$. By \autoref{cubic}, the $\mathcal{N}_{h}$ for $h=1$ and $3$,  satisfies the theorem.

\begin{figure}
\centering
\begin{tikzpicture}
\draw (-2, 0)--++(1, 0)--++(0, -1)--++(-1, 0);
\draw (-3, 0)--++(1, 0)--++(0, -1)--++(-1, 0);
\draw (-4, -0.5)--++(1, 0.5)--++(0, -1)--cycle;

\draw (2, 0)--++(-1, 0)--++(0, -1)--++(1, 0);
\draw (3, 0)--++(-1, 0)--++(0, -1)--++(1, 0);
\draw (4, -0.5)--++(-1, 0.5)--++(0, -1)--cycle;

\draw[dashed] (-1, 0)--(1, 0);
\draw[dashed] (-1, -1)--(1, -1);
\draw (-4, -0.5).. controls+(up:2cm) and +(up:2cm).. (4, -0.5);

\node[circle, inner sep =1, fill, draw] () at (-3, 0) {};
\node[circle, inner sep =1, fill, draw] () at (-2, 0) {};
\node[circle, inner sep =1, fill, draw] () at (-1, 0) {};
\node[circle, inner sep =1, fill, draw] () at (1, 0) {};
\node[circle, inner sep =1, fill, draw] () at (2, 0) {};
\node[circle, inner sep =1, fill, draw] () at (3, 0) {};
\node[circle, inner sep =1, fill, draw] () at (-3, -1) {};
\node[circle, inner sep =1, fill, draw] () at (-2, -1) {};
\node[circle, inner sep =1, fill, draw] () at (-1, -1) {};
\node[circle, inner sep =1, fill, draw] () at (1, -1) {};
\node[circle, inner sep =1, fill, draw] () at (2, -1) {};
\node[circle, inner sep =1, fill, draw] () at (3, -1) {};
\node[circle, inner sep =1, fill, draw] () at (-4, -0.5) {};
\node[circle, inner sep =1, fill, draw] () at (4, -0.5) {};
\end{tikzpicture}
\caption{necklace}
\label{neh}
\end{figure}
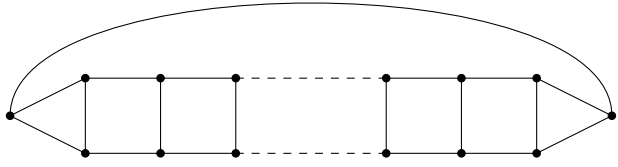

For $h \geq 5$ and $h$ is odd, we color the edge $\{0, 1\}$ by $f$, $\{0, 1'\}$ by $c$. For these edges $\{1, 1'\}, \{2, 2'\}$, $\{3, 3'\}, \dots, \{h, h'\}, \{0, h+1\}$, we  color them by $a$. Next, we color the spine of $T$. At last, we color the edges which are the opposite of the spine of $T$.

If $(h-1) \equiv 0 \pmod{4}$ and $h \geq 5$, we can use the colors in the following order $bcdfbcdf\dots bcdf$ to color the spine of $T$. And we color the edges which are the opposite of the spine of $T$ by $dfbcdfbc\dots dfbc$. At last, we color the edge $\{h, h+1\}$ by $b$ and color the edge $\{h', h+1\}$ by $d$.

If $(h-1) \equiv 2 \pmod{4}$ and $h \geq 5$, we can use the colors in the following order $bcdfbcdf\dots bcdf$ $bcdfbc$ to color the spine of $T$. And we color the edges which are the opposite of the spine of $T$ by $dfbcdfbc\dots dfbcdf$. At last, we color the edge $\{h, h+1\}$ by $d$ and color the edge $\{h', h+1\}$ by $b$.

Now, all the edges of $\mathcal{N}_{h}$ are colored. However, it remains to verify the above is indeed a star $5$-edge-coloring of $\mathcal{N}_{h}$. It is easy to see that the edges incident to the same vertex receive distinct colors, so we only need to show that there is no bichromatic $4$-path in $\mathcal{N}_{h}$. We can notice that the coloring of the edges of $C_{h+2}$ is proper and the coloring of the spine of $T$ is also proper. We also notice that some edges are colored by $a$, while those incident with them are colored by different colors. So, there is no bichromatic 4-path.

Hence, we get a star $5$-edge-coloring of $\mathcal{N}_{h}$ for $h$ is odd. That is if $h$ is odd, $4 \leq \chi'_{\mathrm{star}}(\mathcal{N}_{h}) \leq 5$.
\end{proof}

\subsection{Star chromatic index of complete Halin graphs}
In this section, we will prove the following result.
\begin{theorem}\label{zuidadu6}
Let $G = T \cup C \ (G \neq W_{n})$ be a complete Halin graph with maximum degree $\Delta(G)\geq 6$, then $\chi'_{\mathrm{star}}\leq \lfloor \frac{3\Delta}{2}\rfloor+1$.
\end{theorem}

\begin{lemma}
Let $G = T \cup C \ (G \neq W_{n})$ be a  complete Halin graph. Only one interior vertex adjacent to the leaf is interior vertex, and all other adjacent vertices are leaves.
\end{lemma}
\begin{proof}
Suppose that $d_G(y)\geq4$, $N_G(y)=\{x, z,y_1$,$ \cdots, y_i\}$($i\geq2$) and the interior vertex $y$ is adjacent to at least two interior vertices $x$ and $z$. Let $P$ be a path from the root vertex to the interior vertex $y$, and the path $P$ passes the vertex $z$. Because the vertex $x$ is an interior vertex, there must exist other path from the vertex $x$ to other leaves. Now, we suppose that there is a path from the vertex $x$ to the leaf $w$. Obviously, the path from the root vertex to the leaf $w$ must pass the vertex $x$, $y$ and $z$. Besides, other  adjacent vertices of the vertex $y$ which are leaves must pass the vertex  $y$ and $z$. Thus, the distance from the root vertex to the leaf $w$ is longer than that from the root vertex to other leaves which are adjacent to vertex $y$, a contradiction.
\end{proof}
\begin{proof}[Proof of \autoref{zuidadu6}]
Let $G$ be a graph that satisfies the above conditions. Without loss of generality, we assume that the vertices on the adjacent cycle $C$ counter clockwise, followed by $u_{1, 1}, u_{1, 2}, \cdots, u_{1, j_1}$; $u_{2, 1}, u_{2, 2}, \cdots, u_{2, j_2}; \cdots$; $u_{n, 1}, u_{n, 2}, \cdots, u_{n, j_n}$, and each of these vertices $u_{i, 1}, u_{i, 2}, \cdots, u_{i, j_i}$ is adjacent to the internal vertex $w_i$($1\leq i \leq n$, $j_i\geq 2$).

From \autoref{tree},  we know that the tree of $T$ has a star $ \lfloor \frac {3\Delta}{2}\rfloor$-edge-coloring with colors from set $C_1=\{1, 2, \cdots, \lfloor \frac {3\Delta}{2}\rfloor\}$. Now, to get a star edge-coloring $\psi$ of $G=T\cup C$ using at most $\lfloor \frac{3\Delta}{2}\rfloor+1$ colors from set $C_2=\{1, 2, \cdots, \lfloor \frac {3\Delta}{2}\rfloor, c_0\}$, we consider the coloring of the edges of the adjacent cycle $C$. First, we color the edges $u_{1, j_1}u_{2, 1}, u_{2, j_2}u_{3, 1}, \cdots, u_{n-1, j_{n-1}}u_{n, 1}, u_{n, j_n}u_{1, 1}$ with  color $c_0$.

Let $P_i$ ($1\leq i \leq n$) be the path constituted by the edges $u_{i, 1} u_{i, 2}, u_{i, 2}u_{i, 3}, \cdots, u_{i, j_i-1}u_{i, j_i}$. Next, we color the edges of the path $P_i$.

According to the degree of the internal vertex $w_i$, we consider the following four cases.
\begin{case}
$d_G(w_i)=3$ ($1\leq i \leq n$).
\end{case}
In this case, $P_i$ is a path of length $1$. Now, we just need to color the edge $u_{i, 1}u_{i, 2}$, the set of forbidden colors  for the edge $u_{i, 1}u_{i, 2}$ is $C_\varphi(w_i)\cup \{c_0, \varphi(w_{i-1}u_{i-1, j_{i-1}})$, $ \varphi(w_{i+1}u_{i+1,1})$, $\psi(u_{i-1, j_{i-1}-1}u_{i-1}j_{i-1})$, $\psi(u_{i+1, 1}u_{i+1, 2})\}$  and its size is at most $8$.  Hence, the number of colors available for the edge $u_{i, 1}u_{i, 2}$ is at least $\lfloor \frac{3\Delta}{2}\rfloor+1-8 \geq 2$, thus, the edge $u_{i, 1}u_{i, 2}$ can be colored.
\begin{case}
$d_G(w_i)=4$ ($1\leq i \leq n$).
\end{case}
In this case, $P_i$ is a path of length $2$. Now, we first color the edge $u_{i, 1}u_{i, 2}$, the set of forbidden colors  for the edge $u_{i, 1}u_{i, 2}$ is $C_\varphi(w_i)\cup \{c_0, $ $\varphi(w_{i-1}u_{i-1, j_{i-1}})$, $\psi(u_{i-1, j_{i-1}-1}$ $u_{i-1, j_{i-1}})\}$ and its size is at most $7$. Hence, the number of colors available for the edge $u_{i, 1}u_{i, 2}$ is at least $\lfloor \frac{3\Delta}{2}\rfloor+1-7 \geq 3$, thus, the edge $u_{i, 1}u_{i, 2}$ can be colored.

Next, we color the edge $u_{i, 2}u_{i, 3}$, the set of forbidden colors  for the edge $u_{i, 2}u_{i, 3}$ is $C_\varphi(w_i)\cup \{c_0, \varphi(w_{i+1}u_{i+1, 1})$, $\psi(u_{i, 1}u_{i, 2}), \psi(u_{i+1, 1}u_{i+1, 2})\}$ and its size is at most $8$. Hence, the number of colors available for the edge $u_{i, 2}u_{i, 3}$ is at least $\lfloor \frac{3\Delta}{2}\rfloor+1-8 \geq 2$, thus, the edge $u_{i, 2}u_{i, 3}$ can be colored.
\begin{case}
$d_G(w_i)=5$ ($1\leq i \leq n$).
\end{case}
In this case, $P_i$ is a path of length $3$. Now, we first color the edge $u_{i, 1}u_{i, 2}$, the set of forbidden colors  for the edge $u_{i, 1}u_{i, 2}$ is $C_\varphi(w_i)\cup \{c_0, $ $\varphi(w_{i-1}u_{i-1, j_{i-1}})$, $\psi(u_{i-1, j_{i-1}-1}$ $u_{i-1, j_{i-1}})\}$ and its size is at most $8$. Hence, the number of colors available for the edge $u_{i, 1}u_{i, 2}$ is at least $\lfloor \frac{3\Delta}{2}\rfloor+1-8 \geq 2$, thus, the edge $u_{i, 1}u_{i, 2}$ can be colored.

Next, we color the edge $u_{i, 3}u_{i, 4}$, the set of forbidden colors  for the edge $u_{i, 3}u_{i, 4}$ is $C_\varphi(w_i)\cup \{c_0, \varphi(w_{i+1}u_{i+1, 1})$, $\psi(u_{i, 1}u_{i, 2}), \psi(u_{i+1, 1}u_{i+1, 2})\}$ and its size is at most $9$.  Hence, the number of colors available for the edge $u_{i, 3}u_{i, 4}$ is  at least $\lfloor \frac{3\Delta}{2}\rfloor+1-9 \geq 1$, thus, the edge $u_{i, 3}u_{i, 4}$ can be colored.

Finally, we color the edge $u_{i, 2}u_{i, 3}$, the set of forbidden colors  for the edge  $u_{i, 2}u_{i, 3}$ is $C_\varphi(w_i)\cup \{c_0$, $\psi(u_{i, 1}u_{i, 2}), \psi(u_{i, 3}u_{i, 4})\}$ and its size is at most $8$.  Hence, the number of colors available for the edge $u_{i, 2}u_{i, 3}$ is at least $\lfloor \frac{3\Delta}{2}\rfloor+1-8 \geq 2$, thus, the edge $u_{i, 2}u_{i, 3}$ can be colored.
\begin{case}
$d_G(w_i)\geq 6$ ($1\leq i \leq n$).
\end{case}
In this case, $P_i$ is a path of length at least $4$.
Now,  we first color the edge $u_{i, j_i-3}u_{i,j_i-2}$ with color $c_0$.

Then, we color the edge $u_{i, 1}u_{i, 2}$, the set of forbidden colors  for the edge $u_{i, 1}u_{i, 2}$ is  $C_\varphi(w_i)\cup \{c_0, $ $\varphi(w_{i-1}u_{i-1, j_{i-1}})$, $\psi(u_{i-1, j_{i-1}-1}$ $u_{i-1, j_{i-1}})\}$ and its size is at most $\Delta+3$. Hence, the number of colors available for the edge $u_{i, 1}u_{i, 2}$ is at least $\lfloor \frac{3\Delta}{2}\rfloor+1-\Delta-3= \lfloor \frac{\Delta}{2}\rfloor-2\geq 1$, thus, the edge $u_{i, 1}u_{i, 2}$ can be colored.

Next, we color the edge $u_{i, j_i-1}u_{i, j_i}$, the set of forbidden colors  for the edge $u_{i, j_i-1}u_{i, j_i}$ is $C_\varphi(w_i)\cup \{c_0, \varphi(w_{i+1}$ $u_{i+1, 1})$, $\psi(u_{i+1, 1}u_{i+1, 2})\}$ and its size is at most $\Delta+3$. Hence, the number of colors available for the edge $u_{i, j_i-1}u_{i, j_i}$ is at least $\lfloor \frac{3\Delta}{2}\rfloor+1-\Delta-3= \lfloor \frac{\Delta}{2}\rfloor-2\geq 1$, thus, the edge $u_{i, j_i-1}u_{i, j_i}$ can be colored.

After that, we color the edge $u_{i, 2}u_{i, 3}$, the set of forbidden colors  for the edge $u_{i, 2}u_{i, 3}$ is  $C_\varphi(w_i)\cup \{c_0$, $\psi(u_{i, 1}u_{i, 2})\}$ and its size is at most $\Delta+2$.  Hence, the number of colors available for the edge $u_{i, 2}u_{i, 3}$ is at least $\lfloor \frac{3\Delta}{2}\rfloor+1-\Delta-2=\lfloor \frac{\Delta}{2}\rfloor-1 \geq 2$, thus, the edge $u_{i, 2}u_{i, 3}$ can be colored.

Finally, we color the edges $u_{i,m}u_{i,m+1}$($3 \leq m \leq j_i-2$) and the edge $u_{i,j_i-2}u_{i, j_i-1}$ in counterclockwise order. The set of forbidden colors  for every edge is at most $\Delta+3$, and the number of colors available for every edge is at least $\lfloor \frac{3\Delta}{2}\rfloor+1-\Delta-3=\lfloor \frac{\Delta}{2}\rfloor-2 \geq 1$. Thus, these edges all can be colored. Clearly, there is no bi-chromatic $4$-paths.

From what has been discussed above, the edges on the path $P_i$ all are colored, besides, there is no bicolored $4$-path. Thus, we get a star edge-coloring $\psi$ of graph $G$ using at most $\lfloor \frac{3\Delta}{2}\rfloor+1$ colors.
\resetcounter
\end{proof}
\section{Star chromatic index of square of paths and cycles}
\subsection{Star chromatic index of $P_n^2$}
\begin{lemma}
Let $F_{3}$ be a simple graph as in \autoref{H1}. Then $\chi'_{\mathrm{star}}(F_{3})\geq6$.
\end{lemma}

\begin{figure}
\centering
\begin{tikzpicture}
\coordinate (A) at (0,0);
\coordinate (B) at (2,0);
\coordinate (C) at (0,-1);
\coordinate (D) at (2, -1);
\coordinate (E) at (0, -2);
\draw[line join=round] (A)node[left]{\small$v_{1}$}--(B)node[right]{\small$v_{2}$}--(C)node[left]{\small$v_{3}$}--(D)node[right]{\small$v_{4}$}--(E)node[left]{\small$v_{5}$}--cycle;
\draw[line cap=round] (B)--(D);
\node[circle, inner sep =1, fill, draw] () at (A) {};
\node[circle, inner sep =1, fill, draw] () at (B) {};
\node[circle, inner sep =1, fill, draw] () at (C) {};
\node[circle, inner sep =1, fill, draw] () at (D) {};
\node[circle, inner sep =1, fill, draw] () at (E) {};
\end{tikzpicture}
\caption{Fan $F_{3}$}
\label{H1}
\end{figure}

\begin{proof}
Suppose that $\phi$ is a star $5$-edge-coloring of $F_{3}$ which uses coloring set $\{a, b, c, d, e\}$.Without loss of generalization, assume that $\phi(v_ 1v_3)=a, \phi(v_2v_3)=b, \phi(v_3v_4)=c, \phi(v_3v_5)=d$. The set of available colors for $v_2v_4$ is $\{a, d, e\}$. We will consider the following three cases according to the color of $v_{2}v_{4}$.
\begin{case}
$\phi(v_2v_4)=a$.
\end{case}
First, we color the edge $v_{1}v_{2}$. To avoiding bichromatic $4$-paths and $4$-cycles, $v_{1}v_{2}$ must be colored with $e$. But whatever $v_{4}v_{5}$ is colored, the coloring cannot be a proper edge-coloring, or there exists a bichromatic $4$-path or $4$-cycle. 
\begin{case}
$\phi(v_2v_4)=d$.
\end{case}
Similarly, the edge $v_{1}v_{2}$ must be colored with $e$. So $v_{4}v_{5}$ cannot be properly colored, and we cannot avoid exist a bichromatic $4$-path or $4$-cycle. 
\begin{case}
$\phi(v_2v_4)=e$.
\end{case}
Whether $v_{1}v_{2}$ is colored with $c$ or $d$, $v_{4}v_{5}$ cannot be colored as desired. 
\resetcounter
\end{proof}
\begin{theorem}
For the graph $P_n^2$ with $n \geq 3$,
$$\chi'_{\mathrm{star}}(P_{n}^{2})=
\begin{cases}
3, \hspace*{2mm}n = 3;\\
4, \hspace*{2mm}n = 4;\\
6, \hspace*{2mm}n \geq 5.
\end{cases}$$
\end{theorem}

\begin{proof}
It is easy to check that $\chi_{\mathrm{star}}'(P_3^2)=\chi_{\mathrm{star}}'(C_3)=3$ and $\chi_{\mathrm{star}}'(P_4^2)=4$.

For $n\geq5$, $P_n^2$ contains a subgraph isomorphic to the graph $F_{3}$ (see \autoref{H1}). It is easy to see that $\chi_{\mathrm{star}}'(P_n^2)\geq \chi_{\mathrm{star}}'(F_{3})\geq6$. Next, we want to prove $\chi_{\mathrm{star}}'(P_n^2)\leq 6$, we may assume that $P_n^2$ can be edge-partitioned into two graphs $F$ and $H$. Let $F = P_n^2 - \{v_iv_{i+1} : 2 \leq i \leq n-1 \mbox{and $i$ is even}\}$, $H = \{v_iv_{i+1} : 2 \leq i \leq n-1 \mbox{and $i$ is even}\}$. By \autoref{partition}, we derive that $$\chi_{\mathrm{star}}'(P_n^2)\leq \chi_{\mathrm{star}}'(F)+\chi_{s}'(H|_{P_n^2}).$$

We know that the graph $F$ is a subgraph of $P_2\Box P_m$ for some $m \geq 3$. By \autoref{PXP}, $\chi_{\mathrm{star}}'(P_2\Box P_m)=4$. Thus, we get that  $\chi_{\mathrm{star}}'(F)\leq \chi_{\mathrm{star}}'(P_2\Box P_m)=4$. Besides, $\chi_{s}'(H|_{P_n^2})=2$ is obvious. Therefore, $\chi_{\mathrm{star}}'(P_n^2)\leq \chi_{\mathrm{star}}'(F)+\chi_{s}'(H|_{P_n^2})=4+2=6$.

In summary, $\chi_{\mathrm{star}}'(P_n^2)=6$ for $n \geq 5$.
\end{proof}

\subsection{Star chromatic index of $C_n^2$}
In order to study the star chromatic index of cyclic square graphs, we first consider it as the union of a Hamiltonian cycle and an inner cycle, then dye the star edge chromatic number of cyclic square graphs in turn.
\begin{theorem}
For any cyclic square graphs $C_n^2$, if $n$ is even, then $\chi_{\mathrm{star}}'(C_n^2) \leq9$; if $n$ is odd ($n\neq5$ or $11$), then $\chi_{\mathrm{star}}'(C_n^2) \leq 8$.
\end{theorem}
\begin{proof}
First of all, we use $C_n$ to represent the Hamiltonian cycle in the cyclic square graph. Due to the particularity of its construction, we  consider the following two cases about the star edge-coloring of $C_n^2$.
\begin{case}
$n$ is an even number.
\end{case}

In this case, we first consider the star edge-coloring of Hamiltonian cycle, then consider other edges. Obviously, it suffices to get a star $3$-edge-coloring for the Hamiltonian cycle. Besides, the interior of the Hamiltonian cycle $C_n$ contains two cycles of length $n/2$. Thus, for $n\neq10$, the number of colors required for the star edge-coloring of the inner two cycles is at most $6$. Therefore, $\chi_{\mathrm{star}}'(C_n^2) \leq 3+6=9$. For $n = 10$, see \autoref{n10}, we can get a star $9$-edge-coloring.

\begin{figure}%
\centering
\subcaptionbox{A star $8$-edge-coloring of $C_{10}^{2}$\label{n10}}[0.3\linewidth]
{
\begin{tikzpicture}
\foreach \ang in {0,...,9}
{
\def\pointname{A\ang}
\coordinate (\pointname) at ($(\ang*360/10:2)$);
\node[circle, inner sep =1, fill, draw] () at (A\ang) {};}
\draw (A0)--node {\footnotesize 1}(A1)--node {\footnotesize 2}(A2)--node {\footnotesize 3}(A3)--node {\footnotesize 1}(A4)--node {\footnotesize 2}(A5)--node {\footnotesize 3}(A6)--node {\footnotesize 1}(A7)--node {\footnotesize 2}(A8)--node {\footnotesize 3}(A9)--node {\footnotesize 2}cycle;
\draw (A0)--node {\footnotesize 4}(A2)--node {\footnotesize 5}(A4)--node {\footnotesize 6}(A6)--node {\footnotesize 7}(A8)--node {\footnotesize 5}cycle;
 \draw (A1)--node {\footnotesize 7}(A3)--node {\footnotesize 8}(A5)--node {\footnotesize 4}(A7)--node {\footnotesize 8}(A9)--node {\footnotesize 6}cycle;
\end{tikzpicture}
}
\subcaptionbox{A star $7$-edge-coloring of $C_{7}^{2}$\label{n7}}[0.3\linewidth]
{
\begin{tikzpicture}
\foreach \ang in {0,...,6}
{
\def\pointname{A\ang}
\coordinate (\pointname) at ($(\ang*360/7:2)$);
\node[circle, inner sep =1, fill, draw] () at (A\ang) {};}
\draw (A0)--node {\footnotesize 1}(A1)--node {\footnotesize 2}(A2)--node {\footnotesize 3}(A3)--node {\footnotesize 4}(A4)--node {\footnotesize 5}(A5)--node {\footnotesize 6}(A6)--node {\footnotesize 7}cycle;
\draw (A0)--node {\footnotesize 4}(A2)--node {\footnotesize 6}(A4)--node {\footnotesize 1}(A6)--node {\footnotesize 3}(A1)--node {\footnotesize 5}(A3)--node {\footnotesize 7}(A5)--node {\footnotesize 2}cycle;
\end{tikzpicture}
}
\subcaptionbox{A star $9$-edge-coloring of $C_{11}^{2}$\label{n11}}[0.3\linewidth]
{
\begin{tikzpicture}
\foreach \ang in {0,...,10}
{
\def\pointname{A\ang}
\coordinate (\pointname) at ($(\ang*360/11:2)$);
\node[circle, inner sep =1, fill, draw] () at (A\ang) {};}
\draw (A0)--node {\footnotesize 1}(A1)--node {\footnotesize 2}(A2)--node {\footnotesize 3}(A3)--node {\footnotesize 1}(A4)--node {\footnotesize 2}(A5)--node {\footnotesize 3}(A6)--node {\footnotesize 1}(A7)--node {\footnotesize 2}(A8)--node {\footnotesize 3}(A9)--node{\footnotesize 1}(A10)--node {\footnotesize 4}cycle;
\draw (A0)--node {\footnotesize 5}(A2)--node {\footnotesize 4}(A4)--node {\footnotesize 6}(A6)--node {\footnotesize 4}(A8)--node {\footnotesize 7}(A10)--node {\footnotesize 6}(A1)--node {\footnotesize 7}(A3)--node {\footnotesize 8}(A5)--node {\footnotesize 5}(A7)--node {\footnotesize 8}(A9)--node {\footnotesize 9}cycle;
\end{tikzpicture}
}
\caption{}
\label{}
\end{figure}

\begin{case}
$n$ is an odd number.
\end{case}

In this case, the interior of the Hamiltonian cycle $C_n$ contains one cycle $C_n'$ of length $n$. Thus, the cyclic square graph $C_n^2$ can be divided into two subgraphs $C_n$ and $C_n'$. For $n \neq 7$ or $11$, the restricted-strong chromatic index of $C_n$ on $C_n^2$: $\chi_{s}'(C_n|_{C_n^2})=5$. Obviously, $\chi_{\mathrm{star}}'(C_n')=3$ ($n\neq5$). By \autoref{partition}, we derive that $$\chi_{\mathrm{star}}'(C_n^2) \leq \chi_{s}'(C_n|_{C_n^2})+\chi_{\mathrm{star}}'(C_n')=5+3=8 (n\neq5, 7, 11). $$

For $n=5$, it is easy to check that $\chi_{\mathrm{star}}'(C_5^2)=\chi_{\mathrm{star}}'(K_5)=9$. For $n=7$ and $n=11$, see \autoref{n7} and \ref{n11}, we can get a star $7$-edge-coloring for $C_7^2$ and a star $9$-edge-coloring for $C_{11}^2$.
\resetcounter
\end{proof}

\section{Star edge-coloring of generalized Petersen graphs $P(3n,n)$}

\begin{figure}%
\centering
\subcaptionbox{$i\equiv 0\pmod{2}$, $0\leq i\leq n-1$\label{ieven}}[0.4\linewidth]
{
\begin{tikzpicture}
\foreach \ang in {0, 1, 2}
{
\def\pointnameA{A\ang}
\def\pointnameB{B\ang}
\coordinate (\pointnameA) at ($(\ang*360/3+90:1.2)$);
\node[circle, inner sep =1, fill, draw] () at (A\ang) {};
\coordinate (\pointnameB) at ($(\ang*360/3+90:2.1)$);
\fill (B\ang) circle (2pt);}
\draw (A0)--node{\footnotesize 4}(A1)--node{\footnotesize 3}(A2)--node{\footnotesize 2}cycle;
\draw (A0)--node{\footnotesize 5}(B0);
\draw (A1)--node{\footnotesize 5}(B1);
\draw (A2)--node{\footnotesize 1}(B2);
\fill (A0) node[right] {\footnotesize$v_{i}$};
\fill (A1) node[left] {\footnotesize$v_{i+n}$};
\fill (A2) node[right] {\footnotesize$v_{i+2n}$};
\fill (B0) node[right] {\footnotesize$u_{i}$};
\fill (B1) node[left] {\footnotesize$u_{i+n}$};
\fill (B2) node[right] {\footnotesize$u_{i+2n}$};
\end{tikzpicture}
}
\subcaptionbox{$i\equiv 1\pmod{2}$, $0\leq i\leq n-1$\label{iodd}}[0.4\linewidth]
{
\begin{tikzpicture}
\foreach \ang in {0, 1, 2}
{
\def\pointnameA{A\ang}
\def\pointnameB{B\ang}
\coordinate (\pointnameA) at ($(\ang*360/3+90:1.2)$);
\node[circle, inner sep =1, fill, draw] () at (A\ang) {};
\coordinate (\pointnameB) at ($(\ang*360/3+90:2.1)$);
\fill (B\ang) circle (2pt);}
\draw (A0)--node{\footnotesize 5}(A1)--node{\footnotesize 2}(A2)--node{\footnotesize 1}cycle;
\draw (A0)--node{\footnotesize 4}(B0);
\draw (A1)--node{\footnotesize 4}(B1);
\draw (A2)--node{\footnotesize 3}(B2);
\fill (A0) node[right] {\footnotesize$v_{i}$};
\fill (A1) node[left] {\footnotesize$v_{i+n}$};
\fill (A2) node[right] {\footnotesize$v_{i+2n}$};
\fill (B0) node[right] {\footnotesize$u_{i}$};
\fill (B1) node[left] {\footnotesize$u_{i+n}$};
\fill (B2) node[right] {\footnotesize$u_{i+2n}$};
\end{tikzpicture}
}
\caption{}
\label{}
\end{figure}

\begin{figure}
\centering
\subcaptionbox{}[0.4\linewidth]
{
\begin{tikzpicture}
\foreach \ang in {0, 1, 2}
{
\def\pointnameA{A\ang}
\def\pointnameB{B\ang}
\def\pointnameC{C\ang}
\def\pointnameD{D\ang}
\coordinate (\pointnameA) at ($(\ang*360/3+90:1.2)$);
\node[circle, inner sep =1, fill, draw] () at (A\ang) {};
\coordinate (\pointnameB) at ($(\ang*360/3+90:2.1)$);
\fill (B\ang) circle (2pt);
\coordinate (\pointnameC) at ($(\ang*360/3+90-20:2.6)$);
\fill (C\ang) circle (2pt);
\coordinate (\pointnameD) at ($(\ang*360/3+90+20:2.6)$);
\fill (D\ang) circle (2pt);
\draw (B\ang)--(C\ang);
\draw (B\ang)--(D\ang);
}
\draw (A0)--node{\footnotesize a}(A1)--node{\footnotesize b}(A2)--node{\footnotesize c}cycle;
\draw (A0)--node{\footnotesize b}(B0);
\draw (A1)--node{\footnotesize }(B1);
\draw (A2)--node{\footnotesize }(B2);
\end{tikzpicture}
}
\subcaptionbox{}[0.4\linewidth]
{
\begin{tikzpicture}
\foreach \ang in {0, 1, 2}
{
\def\pointnameA{A\ang}
\def\pointnameB{B\ang}
\def\pointnameC{C\ang}
\def\pointnameD{D\ang}
\coordinate (\pointnameA) at ($(\ang*360/3+90:1.2)$);
\node[circle, inner sep =1, fill, draw] () at (A\ang) {};
\coordinate (\pointnameB) at ($(\ang*360/3+90:2.1)$);
\fill (B\ang) circle (2pt);
\coordinate (\pointnameC) at ($(\ang*360/3+90-20:2.6)$);
\fill (C\ang) circle (2pt);
\coordinate (\pointnameD) at ($(\ang*360/3+90+20:2.6)$);
\fill (D\ang) circle (2pt);
\draw (B\ang)--(C\ang);
\draw (B\ang)--(D\ang);
}
\draw (A0)--node{\footnotesize a}(A1)--node{\footnotesize b}(A2)--node{\footnotesize c}cycle;
\draw (A0)--node{\footnotesize d}(B0);
\draw (A1)--node{\footnotesize d}(B1);
\draw (A2)--node{\footnotesize d}(B2);
\end{tikzpicture}
}
\caption{}
\label{h0}
\end{figure}
\subsection{Star edge-coloring of generalized Petersen graphs $P(3n,n)$}

\begin{lemma} \label{yinli}
Let $H_0$ be a simple graph as in \autoref{h0}, we have $\chi'_{\mathrm{star}}(H_0)\geq5$.
\end{lemma}

\begin{proof}
The graph $H_0$ contains a subgraph $H$, and $\chi'_{\mathrm{star}}(H)=4$. We assume that $H_0$ can be star edge colored with four colors. Let $f$ be a star edge-coloring of $H_0$ which uses four colors, say $C_1=\{a, b, c, d\}$. see \autoref{h0}, Obviously, in this case, the other two edges associated with the vertex $v$ have only one color available, that is the color $d$, which is not a proper edge-coloring. By the symmetry, $H_0$ can not be star edge colored with four colors. Next, we assume that the edges that are associated with the vertices of the triangle and are not in the triangle are all colored with color $d$, Obviously, in this case, the other two edges associated with the vertex $v$ have only one color available, that is the color $c$, which is also not a proper edge-coloring.

In summary, $\chi'_{\mathrm{star}}(H_0)\geq5$
\end{proof}
\begin{theorem}
Let $P(3n,n)$ be a generalized Petersen graph. If $n\geq2$, then $\chi'_{\mathrm{star}}(P(3n,n))=5$.
\end{theorem}

\begin{proof}

For $n=2$, it has that $\chi'_{\mathrm{star}}(P(6,2))=5$.

For $n\geq3$,  the graph $P(3n,n)$ contains a subgraph that is isomorphic to $H_0$, thus, $\chi'_{\mathrm{star}}(P(3n,n))\geq\chi'_{\mathrm{star}}(H_0)\geq 5$ by \autoref{yinli}. Therefore, we just want to prove $\chi'_{\mathrm{star}}(P(3n,n))\leq 5$ below. In other words, we are going to get a star $5$-edge-coloring $f$ of $P(3n,n)$, say $C=\{1, 2, 3, 4, 5\}$.

We first divide the edge set of $P(3n,n)$ into three classes. Let $E_1=\{u_iu_{i+1} : i=0, 1, 2, \dots, 3n-1\}$; $E_2=\{v_iv_{i+n} : i=0, 1, 2, \dots, 3n-1\}$; $E_3=\{u_iv_i : i=0, 1, 2, \dots, 3n-1\}$. Then, we define the star edge-coloring $f$ as follows:

First of all, we color the edges in $E_2$. If $i\equiv 0\pmod{2}$ and $0\leq i\leq n-1$, then color the three edges of the circle in which $v_i$ is located with $4, 3, 2$ in the order of subscript from small to large. If $i\equiv 1\pmod{2}$ and $0\leq i\leq n-1$, then color the three edges of the circle in which $v_i$ is located with $5, 2, 1$ in the order of subscript from small to large.

Next, we color the edges in $E_3$. 

When $n$ is an odd number. If $0\leq i\leq n-1$, then we color the edges with $5, 4$ in the order of subscript from small to large. If $n\leq i\leq 2n-1$, then we color the edges with $5, 4$ in the order of subscript from small to large. If $2n\leq i\leq 3n-1$, then we color the edges with $1, 3$ in the order of subscript from small to large.

When $n$ is an even number. If $0\leq i\leq 2n-1$, then we color the edges with $5, 4$ in the order of subscript from small to large. If $2n\leq i\leq 3n-1$, then we color the edges with $1, 3$ in the order of subscript from small to large.

Now, all  edges in $E_2$ and $E_3$ are colored.

The two types of local structures for the coloring in $E_2$ and $E_3$ are shown in \autoref{ieven} and \ref{iodd}. Besides, there exists no bichromatic paths or cycles of length $4$. But, there exists some paths of length $3$, i.e. $5, 4, 5-$path and $4, 5, 4-$path.

Finally, we are to color the edges in $E_1$ according to the following six cases.

\begin{case}
$n\equiv 0 \pmod{3}$ and $n$ is an odd number.
\label{a1}
\end{case}

If $0\leq i\leq 2n-4$, then we color the edges with $3, 2, 1$ in the order of subscript from small to large, and let $f(u_{2n-3} u_{2n-2})=2$, $f(u_{2n-2} u_{2n-1})=3$, $f(u_{2n-1} u_{2n})=2$. If $2n\leq i\leq 3n-1$, then we color the edges with $5, 4, 2$ in the order of subscript from small to large.

Thus, all  edges of $P(3n,n)$($n\geq2$) are colored. To test whether such a coloring exists bichromatic paths or cycles of length $4$, we discuss the following four cases.

First, we check whether exists a bichromatic path of length $4$ in $E_2\cup E_3$. If such a path exists, it must contain two edges of the $3$-cycle. By observing the local star edge-coloring of $E_2\cup E_3$ in \autoref{ieven} and \ref{iodd}, we can see that such a bichromatic path of length $4$ does not exist.

Secondly, we check whether exists a bichromatic path of length $4$ in  $E_1\cup E_2$. $E_1\cup E_2$ is a graph that contains a $3n$-cycle and $n$ $3$-cycles. When we colored the edges in $E_1$, it was guaranteed that the coloring is a star edge-coloring. In addition, there exist no paths of length $4$ in the $3$-cycle. Thus, it is impossible to exist such a bichromatic path of length $4$.

Thirdly, we check whether exists a bichromatic path of length $4$ in $E_1\cup E_3$. When we colored the edges in $E_1$, it was guaranteed that the coloring is a star edge-coloring. If such a path exists, it must contain one edge of $E_3$. For $0\leq i\leq2n-4$ and $2n\leq i\leq 3n-1$, the colors of the edges whose distance less than or equal to $2$ in $E_1$ are different. So, there exists no bichromatic paths or cycles of length $4$. Then, we check the local coloring of each interface. Although $f(u_{2n-3} u_{2n-2})=2$, $f(u_{2n-2} u_{2n-1})=3$ and $f(u_{2n-1} u_{2n})=2$, the edges $u_{2n}u_{2n}$ in $E_3$ associated with vertex $u_{2n}$ are colored by 1, the edges $u_{2n-3}v_{2n-3}$ in $E_3$ associated with vertex $u_{2n-3}$ are colored by 5. Besides, $f(u_{3n-1} u_{0})=1$, $f(u_{0} u_{1})=3$ and $f(u_{1} u_{2})=2$, the edge $u_{3n-1}v_{3n-1}$ in $E_3$ associated with vertex $u_{3n-1}$ is colored by 1, the edge $v_2v_2$ in $E_3$ associated with vertex $u_2$ is colored by 5, so there exist no bichromatic paths or cycles of length $4$.

Finally, we check whether exists a bichromatic path or cycle of length $4$ in $E_1\cup E_2\cup E_3$. Through our verification one by one, there exist no bichromatic paths or cycles of length $4$.

\begin{case}
$n\equiv 0 \pmod{3}$ and $n$ is an even number.
\end{case}

If $ 0\leq i\leq 2n-1$, then we color the edges with $1, 3, 2, 3$ in the order of subscript from small to large. If $2n\leq i\leq 3n-1$, then we color the edges with $5, 4, 2$ in the order of subscript from small to large.

About the verification of bichromatic paths or cycles of length $4$ is similar to case \ref{a1}, we can get that bichromatic paths or cycles of length $4$ do not exist.

\begin{case}
$n\equiv 1 \pmod{3}$ and $n$ is an odd number.
\end{case}

If $0\leq i\leq n-2$, then we color the edges with $2, 1, 3$ in the order of subscript from small to large, and let $f(u_{n-1}u_n)=1$. If $n\leq i\leq 2n-3$, then we color the edges with $2, 1, 3$ in the order of subscript from small to large, and let $f(u_{2n-2}u_{2n-1})=2$, $f(u_{2n-1}u_{2n})=3$. If $2n\leq i\leq 3n-3$, then we color the edges with $4, 5, 2$ in the order of subscript from small to large, and let $f(u_{3n-2}u_{3n-1})=4$ and $f(u_{3n-1}u_0)=3$.

About the verification of bichromatic paths or cycles of length $4$ is similar to case \ref{a1}, we can get that bichromatic paths or cycles of length $4$ do not exist. 

\begin{case}
$n\equiv 1 \pmod{3}$ and $n$ is an even number.
\end{case}

If $0\leq i\leq2n-1$, then we color the edges with $2, 3, 1$ in the order of subscript from small to large. If $2n\leq i\leq 3n-2$, then we color the edges with $5, 2, 4$ in the order of subscript from small to large, and let $f(u_{3n-1}u_0)=1$.

About the verification of bichromatic paths or cycles of length $4$ is similar to case \ref{a1}, we can get that bichromatic paths or cycles of length $4$ do not exist.
\begin{case}
$n\equiv 2 \pmod{3}$ and $n$ is an odd number.
\end{case}

If $0\leq i\leq2n-1$, then we color the edges with $3, 1, 2$ in the order of subscript from small to large. If $2n\leq i\leq 3n-1$, then we color the edges with $4, 2, 5$ in the order of subscript from small to large.

About the verification of bichromatic paths or cycles of length $4$ is similar to case \ref{a1}, we can get that bichromatic paths or cycles of length $4$ do not exist.
\begin{case}
$n\equiv 2 \pmod{3}$ and $n$ is an even number.
\end{case}

If $0\leq i\leq2n-2$, then we color the edges with $1, 3, 2$ in the order of subscript from small to large, and let $f(u_{2n-1}u_{2n})=3$. If $2n\leq i\leq 3n-4$, then we color the edges with $5, 2, 4$ in the order of subscript from small to large, and let $f(u_{3n-3}u_{3n-2})=5$,  $f(u_{3n-2}u_{3n-1})=4$ and $f(u_{3n-1}u_0)=2$.

About the verification of bichromatic paths or cycles of length $4$ is similar to case \ref{a1}, we can get that bichromatic paths or cycles of length $4$ do not exist.

In summary, $\chi'_{\mathrm{star}}(P(3n,n))=5$.
\end{proof}

\end{document}